\documentclass{article}

\usepackage{amsmath}
\usepackage{amssymb}
\usepackage{amsthm}
\usepackage{dsfont}
\usepackage{mathrsfs}
\usepackage{wasysym}
\usepackage{enumitem}
\usepackage{mathtools}
\usepackage{stmaryrd}
\usepackage{caption}
\usepackage{subcaption}
\usepackage{float}
\usepackage[nottoc]{tocbibind}	
\usepackage[dvipsnames]{xcolor}
\usepackage[a4paper, left=2.7cm, right=2.7cm, top=3.5cm, bottom=2cm]{geometry}
\usepackage[pdftex=true,hyperindex=true,pdfborder={0 0 0}]{hyperref}
\usepackage{doi}
\usepackage{tikz}
\usepackage{array}
\usetikzlibrary{shapes,arrows,automata,matrix,fit}

\usepackage[utf8]{inputenc}  
\usepackage[T1]{fontenc} 

\theoremstyle{plain}
\newtheorem{theorem}{Theorem}[section]

\newtheorem{proposition}[theorem]{Proposition}

\theoremstyle{definition}
\newtheorem{remark}[theorem]{Remark}

\newcommand{\N}{\mathbb N}	
\newcommand{\sgn}{\mathrm{sgn}}
\newcommand{\Z}{\mathbb Z}	
\newcommand{\E}{\mathbb E}	
\renewcommand{\P}{\mathbb P}	
	
\newcommand{\G}{\mathcal G}

\newtheorem{conjecture}[theorem]{Conjecture}

\begin{document}

\begin{flushleft}
\LARGE \textbf{Corner percolation with preferential directions}\\
\vspace{0,5cm}
\large \textbf{Régine Marchand, Irène Marcovici and Pierrick Siest}
\end{flushleft}

\begin{flushleft}

Université de Lorraine, IECL\\
54000 Nancy\\
E-mail adress: \texttt{regine.marchand@univ-lorraine.fr, irene.marcovici@univ-lorraine.fr, pierrick.siest@univ-lorraine.fr}\\
\vspace{0,5cm}

\end{flushleft}

\noindent \textbf{Abstract.} Corner percolation is a dependent bond percolation model on ${\mathbb Z}^2$ introduced by B\'alint T\'oth, in which each vertex has exactly two incident edges, perpendicular to each other. G\'abor Pete has proven in 2008 that under the maximal entropy probability measure, all connected components are finite cycles almost surely. We consider here a regime where West and North directions are preferred with probability $p$ and $q$ respectively, with $(p,q)\ne(\frac{1}{2},\frac{1}{2})$. We prove that there exists almost surely an infinite number of infinite connected components, which are in fact infinite paths. Furthermore, they all have the same asymptotic slope $\frac{2q-1}{1-2p}$.

\section{Introduction}

Various constrained percolation models on $\Z^2$ have been studied, including models with restrictions on the degree of each vertex, see for example~\cite{PercoEulerienne} or~\cite{ConstrainedDegreePercoSidoravicius}. In the present work, we study an even more constrained percolation model, called \emph{corner percolation}, where each vertex has exactly two incident open edges, perpendicular to each other. It was introduced by B\'alint T\'oth, and studied by G\'abor Pete in the maximal entropy regime~\cite{Corner}. 

Corner percolation configurations are also known under the name of \emph{hitomezashi} design, by analogy with a Japanese style of embroidery called \emph{sashiko}, that creates patterns satisfying the same constraints, see Figure \ref{fig:config}. Since at each vertex, there is exactly one horizontal and one vertical open edge, the components are either finite cycles or bi-infinite paths, made in both case of a perfect alternation of horizontal and vertical edges. The mathematical properties of hitomezashi loops have been recently investigated and still raise many questions. In particular, it was proven that their length is always congruent to $4$ modulo $8$, but the proof is quite intricate, and the enumeration of hitomezashi loops according to their length is still an open question~\cite{Hitomezashi}.

Although they exhibit a rich combinatorial structure, corner percolation configurations can be created by a very simple procedure. Indeed, in order to satisfy the constraint, there must be on each horizontal line a perfect alternation of open and closed edges, and the same holds for vertical lines. Consequently, a configuration can be parametrized by two binary sequences specifying, for each horizontal and vertical line, which edges are open. A natural question is then to ask whether there exists an infinite connected component, depending on the probability distributions of the two binary sequences. If the choice is made uniformly at random and independently for each horizontal and vertical line, it can be interpreted as the maximal entropy probability measure on all configurations satisfying the constraints. G\'abor Pete treated this case and proved that under the maximal entropy probability measure, all connected components are finite cycles almost surely~\cite{Corner}. 

Let us give an alternative interpretation of the maximal entropy probability measure. We consider that each horizontal line of the grid $\Z^2$ is a one-way road that can be either oriented to the East or to the West with probability $\frac{1}{2}$, and the same for vertical lines with North and South. This orientation is fixed once and for all. We start from the origin of the grid and follow the horizontal road according to its orientation. At each corner, we turn left or right according to the direction of the road encountered, see Figure \ref{fig:roads}. 
This process describes a deterministic walk in a random environment, and it can be seen that the path taken follows the component of the origin in a corner percolation configuration distributed under the maximal entropy probability measure. Consequently, Pete's result ensures that the trajectory necessarily comes back where it started.

\begin{figure}[h]

\begin{minipage}[c]{0.42\linewidth}
    \centering
    \includegraphics[width=6cm,height=6cm]{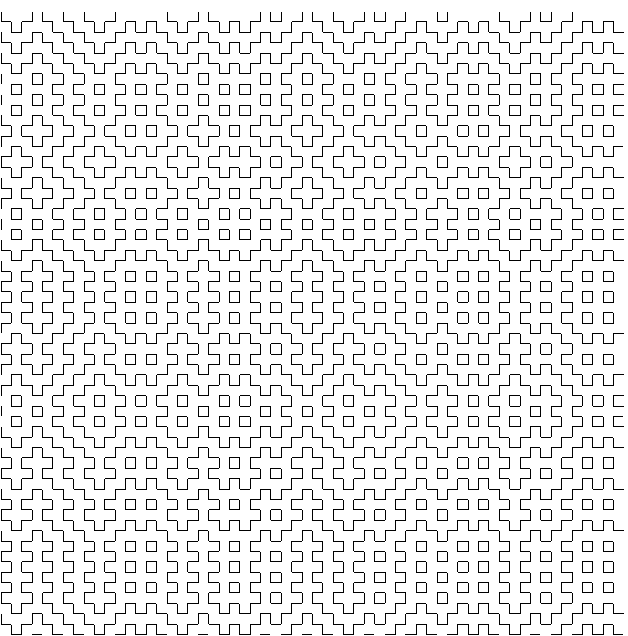}
    \caption{A configuration of corner percolation.}
    \label{fig:config}
\end{minipage}
\hfill
\begin{minipage}[c]{.48\linewidth}
    \centering
    
    \begin{tikzpicture}[scale=0.6]
    
    \draw[black, thin, ->] (0,-4) -- (0,4) node[anchor=west]{$y$};
    \draw[black, thin, ->] (-4,0) -- (4,0) node[anchor=west]{$x$};
    
    \draw[blue, thin,<-] (-4,1) -- (-3.5,1) node[anchor=west]{};
    \draw[blue, thin,->] (-4,2) -- (-3.5,2)
    node[anchor=west]{};
    \draw[blue, thin,->] (-4,3) -- (-3.5,3)
    node[anchor=west]{};
    \draw[blue, thin,->] (-4,-1) -- (-3.5,-1) node[anchor=west]{};
    \draw[blue, thin,->] (-4,0) -- (-3.5,0) node[anchor=west]{};
    \draw[blue, thin,<-] (-4,-2) -- (-3.5,-2) node[anchor=west]{};
    \draw[blue, thin,<-] (-4,-3) -- (-3.5,-3) node[anchor=west]{};
    \draw[blue, thin,<-] (-4,-4) -- (-3.5,-4) node[anchor=west]{};
    
    \draw[green, thin,<-] (1,-5) -- (1,-4.5) node[anchor=west]{};
    \draw[green, thin,->] (2,-5) -- (2,-4.5) node[anchor=west]{};
    \draw[green, thin,->] (-1,-5) -- (-1,-4.5) node[anchor=west]{};
    \draw[green, thin,<-] (0,-5) -- (0,-4.5) node[anchor=west]{};
    \draw[green, thin,->] (-2,-5) -- (-2,-4.5) node[anchor=west]{};
    \draw[green, thin,->] (-3,-5) -- (-3,-4.5) node[anchor=west]{};
    \draw[green, thin,<-] (3,-5) -- (3,-4.5) node[anchor=west]{};
    \draw[green, thin,<-] (4,-5) -- (4,-4.5) node[anchor=west]{};

    \filldraw[red] (0,0) circle (2pt) node[anchor=north ]{$(0,0)$};
    
    \draw[red, thin, ->] (0,0) -- (1,0) node[anchor=west]{};
    
    \draw[red, thin] (1,0) -- (1,-1) node[anchor=west]{};
    
    \draw[red, thin] (1,-1) -- (2,-1) node[anchor=west]{};
    
    \draw[red, thin] (2,-1) -- (2,0) node[anchor=west]{};
    
    \draw[red, thin] (2,0) -- (3,0) node[anchor=west]{};
    
    \draw[red, thin] (3,0) -- (3,-1) node[anchor=west]{};
    
    \draw[red, thin] (3,-1) -- (4,-1) node[anchor=west]{};
        
    \draw[red, thin] (4,-1) -- (4,-2) node[anchor=west]{};
            
    \draw[red, thin] (4,-2) -- (3,-2) node[anchor=west]{};
                
    \draw[red, thin] (3,-2) -- (3,-3) node[anchor=west]{};
                    
    \draw[red, thin] (3,-3) -- (2,-3) node[anchor=west]{};
                        
    \draw[red, thin] (2,-3) -- (2,-2) node[anchor=west]{};
    
    \draw[red, thin] (2,-2) -- (1,-2) node[anchor=west]{};
    
    \draw[red, thin] (1,-2) -- (1,-3) node[anchor=west]{};
    
    \draw[red, thin] (1,-3) -- (0,-3) node[anchor=west]{};
    
    \draw[red, thin] (0,-3) -- (0,-4) node[anchor=west]{};
    
    \draw[red, thin] (0,-4) -- (-1,-4) node[anchor=west]{};
    
    \draw[red, thin] (-1,-4) -- (-1,-3) node[anchor=west]{};
    
    \draw[red, thin] (-1,-3) -- (-2,-3) node[anchor=west]{};
    
    \draw[red, thin] (-2,-3) -- (-2,-2) node[anchor=west]{};
    
    \draw[red, thin] (-2,-2) -- (-3,-2) node[anchor=west]{};
    
    \draw[red, thin] (-3,-2) -- (-3,-1) node[anchor=west]{};
    
    \draw[red, thin] (-3,-1) -- (-2,-1) node[anchor=west]{};
    
    \draw[red, thin] (-2,-1) -- (-2,0) node[anchor=west]{};
    
    \draw[red, thin] (-2,0) -- (-1,0) node[anchor=west]{};
    
    \draw[red, thin] (-1,0) -- (-1,1) node[anchor=west]{};
    
    \draw[red, thin] (-1,1) -- (-2,1) node[anchor=west]{};
    
    \draw[red, thin] (-2,1) -- (-2,2) node[anchor=west]{};
    
    \draw[red, thin] (-2,2) -- (-1,2) node[anchor=west]{};
    
    \draw[red, thin] (-1,2) -- (-1,3) node[anchor=west]{};
    
    \draw[red, thin] (-1,3) -- (0,3) node[anchor=west]{};
    
    \draw[red, thin] (0,3) -- (0,2) node[anchor=west]{};
    
    \draw[red, thin] (0,2) -- (1,2) node[anchor=west]{};
    
    \draw[red, thin] (1,2) -- (1,1) node[anchor=west]{};
    
    \draw[red, thin] (1,1) -- (0,1) node[anchor=west]{};
    
    \draw[red, thin] (0,1) -- (0,0) node[anchor=west]{};
                
    \end{tikzpicture}
    \caption{An instance of a finite trajectory for the deterministic walk starting from the origin, in an environment giving the directions of the roads. }
    \label{fig:roads}
\end{minipage}
\end{figure}

This point of view leads us to consider a case where some directions can be preferred: we extend this model by considering that each vertical road is oriented to the North with probability $q$, and that each horizontal road is oriented to the West with probability $p$. We prove the following:

\begin{theorem} \label{PremierResultat}
If $(p,q)\ne(\frac{1}{2},\frac{1}{2})$, then with positive probability the path starting from the origin never goes back to its starting point. Moreover, there almost surely exists an infinite number of infinite paths, and they all have a same asymptotic slope, which is equal to~$\frac{2q-1}{1-2p}$ (and if $p=1/2$, the asymptotic slope is infinite).
\end{theorem}


This result  will be proved in two steps, Theorem \ref{existinftypath} and Theorem \ref{directionofpath}.
In the next section, we define the model more formally and we prove its ergodicity.

\section{Definition of the model and ergodicity}

\subsection{Definition and notations}

Let us denote by $(\Z^2,\mathbb{E}^2)$ the two-dimensional grid, that is, the graph whose set of vertices is $\Z^2$ and whose set of edges is $\mathbb{E}^2=\{\{x,x+e_1\}: x\in\Z^2\}\cup \{\{x,x+e_2\}: x\in\Z^2\}$, where $e_1=(1,0)$ and $e_2=(0,1)$.

Let $\G$ be the set of subgraphs $G=(\Z^2,\mathcal{E})$ of $(\Z^2,\mathbb{E}^2)$ such that each vertex has exactly one horizontal and one vertical adjacent edge, that is: for any $x\in\Z^2$, exactly one of the edges among $\{x,x+e_1\}$ and $\{x,x-e_1\}$ belongs to $\mathcal{E}$, and exactly one of the edges among $\{x,x+e_2\}$ and $\{x,x-e_2\}$ belongs to $\mathcal{E}$. Note that this implies that connected components are all infinite simple paths or finite circuits. 



Let us introduce the following sets of \emph{even} and \emph{odd} elements of $\Z^2$:
\begin{align*}
&\Z^2_e =  \{(x_1,x_2) \in \mathbb Z^2:  \, x_1+x_2 \in 2\Z\},\\
&\Z^2_o =  \{(x_1,x_2) \in \mathbb Z^2:  \, x_1+x_2 \in 1+2\Z\}.
\end{align*}
Let $\Omega=\{-1,1\}^{\Z}\times \{-1,1\}^{\Z}$. Our notations are inspired by those of Pete. We define a map $X:\Omega\to\G$ as follows. For $\omega=(\xi,\eta)\in\Omega$, the graph $X(\omega)=(\Z^2,\mathcal{E})$ is defined this way (see Figure \ref{ConstructionModel}): $\mathcal{E}$ is the set of open edges, and for all $i\in\mathbb{Z}$,

\begin{itemize}
    \item if $\xi(i)=1$, then for all $k\in\mathbb{Z}$, the vertical edge $\{(i,k),(i,k)+e_2\}$ is open if and only if $(i,k)\in \mathbb{Z}^2_o$. If $\xi(i)=0$, then for all $k\in\mathbb{Z}$, the vertical edge $\{(i,k),(i,k)+e_2\}$ is open if and only if $(i,k)\in \mathbb{Z}^2_e$;

    \item  if $\eta(i)=1$, then for all $k\in\mathbb{Z}$, the horizontal edge $\{(k,i),(k,i)+e_1\}$ is open if and only if $(k,i)\in \mathbb{Z}^2_o$. If $\eta(i)=0$, then for all $k\in\mathbb{Z}$, the horizontal edge $\{(k,i),(k,i)+e_1\}$ is open if and only if $(k,i)\in \mathbb{Z}^2_e$.
\end{itemize}

\begin{figure}[h]
    \centering
    \begin{tikzpicture}[scale=0.6]
    \draw[black, thick, ->] (0,-4) -- (0,5) node[anchor=west]{$y$};
    \draw[black, thick, ->] (-5,0) -- (5,0) node[anchor=west]{$x$};
    \filldraw[black] (-4,2) circle (0pt) node[anchor=east]{0};
    \filldraw[black] (1,-4) circle (0pt) node[anchor=south]{1};
    \filldraw[black] (-3,-4) circle (0pt) node[anchor=south]{0};
    \filldraw[black] (-4,1) circle (0pt) node[anchor=east]{1};
    
    \draw[blue, thick] (1,-2) -- (1,-1);
        \draw[blue, thick] (1,0) -- (1,1);
            \draw[blue, thick] (1,2) -- (1,3);
    \draw[blue, thick] (-4,2) -- (-3,2);
        \draw[blue, thick] (-2,2) -- (-1,2);
            \draw[blue, thick] (0,2) -- (1,2);
                \draw[blue, thick] (2,2) -- (3,2);
    \draw[blue, thick] (-3,4) -- (-3,3);
        \draw[blue, thick] (-3,2) -- (-3,1);
            \draw[blue, thick] (-3,0) -- (-3,-1);
                \draw[blue, thick] (-3,-2) -- (-3,-3);
    \draw[blue, thick] (-4,1) -- (-3,1);
        \draw[blue, thick] (-2,1) -- (-1,1);
            \draw[blue, thick] (0,1) -- (1,1);
                \draw[blue, thick] (2,1) -- (3,1);

    \filldraw[red] (-3,-2) circle (1.5pt);
        \filldraw[red] (-3,0) circle (1.5pt);
            \filldraw[red] (-3,2) circle (1.5pt);
                \filldraw[red] (-3,4) circle (1.5pt);
    \filldraw[red] (-1,-2) circle (1.5pt);
        \filldraw[red] (-1,0) circle (1.5pt);
            \filldraw[red] (-1,2) circle (1.5pt);
                \filldraw[red] (-1,4) circle (1.5pt);
    \filldraw[red] (1,-2) circle (1.5pt);
        \filldraw[red] (1,0) circle (1.5pt);
            \filldraw[red] (1,2) circle (1.5pt);
                \filldraw[red] (1,4) circle (1.5pt);
    \filldraw[red] (3,-2) circle (1.5pt);
        \filldraw[red] (3,0) circle (1.5pt);
            \filldraw[red] (3,2) circle (1.5pt);
                \filldraw[red] (3,4) circle (1.5pt);
                
    \filldraw[red] (-4,-3) circle (1.5pt);
        \filldraw[red] (-4,-1) circle (1.5pt);
            \filldraw[red] (-4,1) circle (1.5pt);
                \filldraw[red] (-4,3) circle (1.5pt);            
    \filldraw[red] (-2,-3) circle (1.5pt);
        \filldraw[red] (-2,-1) circle (1.5pt);
                \filldraw[red] (-2,1) circle (1.5pt);
                    \filldraw[red] (-2,3) circle (1.5pt);
    \filldraw[red] (0,-3) circle (1.5pt);
        \filldraw[red] (0,-1) circle (1.5pt);
            \filldraw[red] (0,1) circle (1.5pt);
                \filldraw[red] (0,3) circle (1.5pt);
    \filldraw[red] (2,-3) circle (1.5pt);
        \filldraw[red] (2,-1) circle (1.5pt);
            \filldraw[red] (2,1) circle (1.5pt);
                \filldraw[red] (2,3) circle (1.5pt);
    \filldraw[red] (4,-3) circle (1.5pt);
        \filldraw[red] (4,-1) circle (1.5pt);
            \filldraw[red] (4,1) circle (1.5pt);
                \filldraw[red] (4,3) circle (1.5pt);
                
    \filldraw[gray] (-4,-2) circle (1.5pt);
        \filldraw[gray] (-4,0) circle (1.5pt);
            \filldraw[gray] (-4,2) circle (1.5pt);
                \filldraw[gray] (-4,4) circle (1.5pt);
    \filldraw[gray] (-2,-2) circle (1.5pt);
        \filldraw[gray] (-2,0) circle (1.5pt);
            \filldraw[gray] (-2,2) circle (1.5pt);
                \filldraw[gray] (-2,4) circle (1.5pt);
    \filldraw[gray] (0,-2) circle (1.5pt);
        \filldraw[gray] (0,0) circle (1.5pt);
            \filldraw[gray] (0,2) circle (1.5pt);
                \filldraw[gray] (0,4) circle (1.5pt);
    \filldraw[gray] (2,-2) circle (1.5pt);
        \filldraw[gray] (2,0) circle (1.5pt);
            \filldraw[gray] (2,2) circle (1.5pt);
                \filldraw[gray] (2,4) circle (1.5pt);
    \filldraw[gray] (4,-2) circle (1.5pt);
        \filldraw[gray] (4,0) circle (1.5pt);
            \filldraw[gray] (4,2) circle (1.5pt);
                \filldraw[gray] (4,4) circle (1.5pt);
                
    \filldraw[gray] (-3,-3) circle (1.5pt);
        \filldraw[gray] (-3,-1) circle (1.5pt);
            \filldraw[gray] (-3,1) circle (1.5pt);
                \filldraw[gray] (-3,3) circle (1.5pt);
    \filldraw[gray] (-1,-3) circle (1.5pt);
        \filldraw[gray] (-1,-1) circle (1.5pt);
            \filldraw[gray] (-1,1) circle (1.5pt);
                \filldraw[gray] (-1,3) circle (1.5pt);
    \filldraw[gray] (1,-3) circle (1.5pt);
        \filldraw[gray] (1,-1) circle (1.5pt);
            \filldraw[gray] (1,1) circle (1.5pt);
                \filldraw[gray] (1,3) circle (1.5pt);
    \filldraw[gray] (3,-3) circle (1.5pt);
        \filldraw[gray] (3,-1) circle (1.5pt);
            \filldraw[gray] (3,1) circle (1.5pt);
                \filldraw[gray] (3,3) circle (1.5pt);
                
    \end{tikzpicture}
    \caption{Construction of some open edges knowing that $\xi(1)=1$, $\xi(-3)=0$, $\eta(1)=1$ and $\eta(2)=0$. Points of $\mathbb{Z}_e^2$ are colored in gray, points of $\mathbb{Z}_o^2$ in red.}
    \label{ConstructionModel}
\end{figure}

By construction, $X$ takes values in $\G$. Furthermore, the map $X:\Omega \to \G$ is bijective. For $p,q\in ]0,1[$, we introduce the probability 
$$\mu^{p,q}=(q\delta_1+(1-q)\delta_{-1})^{\otimes \mathbb{Z}}\otimes (p\delta_1+(1-p)\delta_{-1})^{\otimes \mathbb{Z}}$$
on $\Omega$ (with the product $\sigma$-algebra). We denote by $\mathbb{P}^{p,q}$ the image of $\mu^{p,q}$ by $X$, and we call $\mathbb{P}^{p,q}$ the \emph{corner percolation model} with parameters $(p,q)$. 
When there is no ambiguity on the parameters, we write $\mu$ (respectively $\P$) instead of $\mu^{p,q}$ (respectively $\P^{p,q}$). In the special case $q=p=\frac{1}{2}$, this distribution can be interpreted as the uniform distribution on $\G$. In this work, we are interested in the properties of random graphs of~$\G$ distributed according to  $\mathbb{P}^{p,q}$, and more specifically in the existence and properties of infinite paths. 

\subsection{Invariance by even translations and ergodicity}

For $v\in \mathbb Z^2$, the translation $\tau_{v}$ is the map: $\tau_{v}:\G\to\G$ defined, for $G=(\Z^2,\mathcal{E})\in\G$, by:
$$\tau_{v}(G)=(\Z^2, v+\mathcal{E}),$$
where $v+\mathcal{E}=\{\{x+v,y+v\}: \{x,y\}\in \mathcal{E}\}.$ Observe that for any $v\in \mathbb Z^2$, the translation~$\tau_{v}$ is bijective, of inverse the translation $\tau_{-v}$. If $v\in \Z^2_e$, we call $\tau_v$ an \emph{even translation}.

\begin{remark} \label{symmetry} 
The law of the corner percolation model has some useful symmetry properties:
\begin{enumerate}
    \item The model of parameters $(p,q)$ has the same law as the model of parameters $(1-p,1-q)$  translated by the vector $(1,0)$ or by the vector $(0,1)$. Therefore, the distribution $\P$ is invariant under even translations. 
    \item The model of parameters $(p,q)$ has the same law as the model obtained by symmetry with respect to the $x$-axis (respectively $y$-axis) of the model of parameters $(p,1-q)$ (respectively $(1-p,q)$).
    \item The model of parameters $(q,p)$ has the same law as the model obtained by symmetry with respect to the line of equation $y=x$ of the model of parameters $(p,q)$.
\end{enumerate}
These transformations also let the number of infinite paths invariant.
\end{remark}


\begin{proposition}\label{ergo}
For all $(k,\ell)\in \mathbb{Z}^2_e$, such that $k\not=0$ and $\ell\not=0$, the dynamical system $(\G,\mathfrak{F},\P,\tau_{(k,\ell)})$ is \emph{ergodic}, that is: for any event $A$ such that $\tau_{(k,\ell)}^{-1}(A)=A$, we have $\P(A)\in\{0,1\}$.
\end{proposition}

\begin{proof} We will show that $(\G,\P,\mathfrak{F},\tau_{(k,\ell)})$ is \emph{mixing}, that is: for any events $A$ and $B$, 
$$\lim\limits_{n\to +\infty}\P(A\cap \tau_{(k,\ell)}^{-n}(B))=\P(A)\P(B).$$ 
Indeed, Corollary 2.2.6 of \cite{Keller} shows that it implies ergodicity.

Let us first assume that $A$ and $B$ depend only on finitely many lines and columns. Let~$I_A$ and $J_A$ (resp. $I_B$ and $J_B$) be finite subsets of $\Z$ such that $A$ (resp. $B$) depends only on $\{\eta(i):{i\in I_A}\}$ and $\{\xi(j):{j\in J_A}\}$ (resp. on $\{\eta(i):{i\in I_B}\}$ and $\{\xi(j):{j\in J_B}\}$). Since $k\not=0$ and $\ell\not=0$, for $n$ large enough, we have $I_A\cap (-kn+I_B)=\emptyset$ and $J_A\cap (-\ell n+J_B)=\emptyset$, so that the events $A$ and $\tau_{(k,\ell)}^{-n}(B)$ depend on disjoint sets of lines and columns. The probability $\mu$ is a product distribution, so we have $\P(A\cap\tau_{(k,\ell)}^{-n}(B))=\P(A)\times\P(\tau_{(k,\ell)}^{-n}(B))$. Since $(k,\ell)\in\Z^2_e$, the translation $\tau_{(k,\ell)}$ is even and preserves $\P$. So, we obtain $$\P(A\cap\tau_{(k,\ell)}^{-n}(B))=\P(A)\P(B).$$ 
Events depending only on finitely many lines and columns generate the product $\sigma$-algebra, so that by Lemma 2.2.7 of \cite{Keller}, we obtain that $(\G,\mathfrak{F},\P,\tau_{(k,\ell)})$ is mixing, and therefore ergodic.
\end{proof}







Now, we state two consequences of Birkhoff's pointwise ergodic theorem, that
will be used multiple times thereafter. We first define
$$\begin{array}{ccccc}
\phi &: & \mathbb{Z}^2 & \to & \mathbb{Z}^2_e \\
 & & x=(x_1,x_2) & \mapsto & (x_1-x_2,x_1+x_2) \\
\end{array},$$
which is a bijection. We also set (see Figure \ref{CNBN}), for all $N \in \mathbb{N}^*$, $B_N=\llbracket -N,N \rrbracket^2$ and 
$$C_N=\phi(B_N)=\{x\in \mathbb{Z}^2_e~ : ~ -2N\leq x_1\leq 2N \text{ and } -2N+|x_1|\leq x_2 \leq 2N-|x_1| \}.$$

\begin{figure}
    \centering
    \begin{tikzpicture}[scale=0.5]
    \def\rayhole{2.5pt}
    \draw[black, thick, ->] (0,-5) -- (0,5) node[anchor=west]{$y$};
    \draw[black, thick, ->] (-5,0) -- (5,0) node[anchor=west]{$x$};
    
    \filldraw[draw=red,fill=white,very thick] (0,0) circle (4pt);
    \draw[draw=red,fill=white,very thick] (0,2) circle (4pt);
    \draw[draw=red,fill=white,very thick] (0,4) circle (4pt);
    \draw[draw=red,fill=white,very thick] (0,-2) circle (4pt);
    \draw[draw=red,fill=white,very thick] (0,-4) circle (4pt);
    
    \draw[draw=red,fill=white,very thick] (1,1) circle (4pt);
    \draw[draw=red,fill=white,very thick] (1,3) circle (4pt);
    \draw[draw=red,fill=white,very thick] (1,-1) circle (4pt);
    \draw[draw=red,fill=white,very thick] (1,-3) circle (4pt);
    
    \draw[draw=red,fill=white,very thick] (2,0) circle (4pt);
    \draw[draw=red,fill=white,very thick] (2,2) circle (4pt);
    \draw[draw=red,fill=white,very thick] (2,-2) circle (4pt);
    
    \draw[draw=red,fill=white,very thick] (3,1) circle (4pt);
    \draw[draw=red,fill=white,very thick] (3,-1) circle (4pt);

    \draw[draw=red,fill=white,very thick] (4,0) circle (4pt);

    \draw[draw=red,fill=white,very thick] (-1,1) circle (4pt);
    \draw[draw=red,fill=white,very thick] (-1,3) circle (4pt);
    \draw[draw=red,fill=white,very thick] (-1,-1) circle (4pt);
    \draw[draw=red,fill=white,very thick] (-1,-3) circle (4pt);
    
    \draw[draw=red,fill=white,very thick] (-2,0) circle (4pt);
    \draw[draw=red,fill=white,very thick] (-2,2) circle (4pt);
    \draw[draw=red,fill=white,very thick] (-2,-2) circle (4pt);
    
    \draw[draw=red,fill=white,very thick] (-3,1) circle (4pt);
    \draw[draw=red,fill=white,very thick] (-3,-1) circle (4pt);
    
    \draw[draw=red,fill=white,very thick] (-4,0) circle (4pt);
    
    \filldraw[black] (0,0) circle (0pt) node[anchor=north east] {$0$};
    \filldraw[black] (2,0) circle (0pt) node[anchor=north east] {$2$};
    \filldraw[black] (0,2) circle (0pt) node[anchor=north east] {$2$};

    \filldraw[gray] (-2,-2) circle (2pt);
    \filldraw[gray] (-2,-1) circle (2pt);
    \filldraw[gray] (-2,0) circle (2pt);
    \filldraw[gray] (-2,1) circle (2pt);
    \filldraw[gray] (-2,2) circle (2pt);
    
    \filldraw[gray] (-1,-2) circle (2pt);
    \filldraw[gray] (-1,-1) circle (2pt);
    \filldraw[gray] (-1,0) circle (2pt);
    \filldraw[gray] (-1,1) circle (2pt);
    \filldraw[gray] (-1,2) circle (2pt);
    
    \filldraw[gray] (0,-2) circle (2pt);
    \filldraw[gray] (0,-1) circle (2pt);
    \filldraw[gray] (0,0) circle (2pt);
    \filldraw[gray] (0,1) circle (2pt);
    \filldraw[gray] (0,2) circle (2pt);
    
    \filldraw[gray] (1,-2) circle (2pt);
    \filldraw[gray] (1,-1) circle (2pt);
    \filldraw[gray] (1,0) circle (2pt);
    \filldraw[gray] (1,1) circle (2pt);
    \filldraw[gray] (1,2) circle (2pt);
    
    \filldraw[gray] (2,-2) circle (2pt);
    \filldraw[gray] (2,-1) circle (2pt);
    \filldraw[gray] (2,0) circle (2pt);
    \filldraw[gray] (2,1) circle (2pt);
    \filldraw[gray] (2,2) circle (2pt);
    
    \end{tikzpicture}
    \caption{Representation of $C_N$ (red circles) and $B_N$ (grey points), for $N=2$.}
    \label{CNBN}
\end{figure}
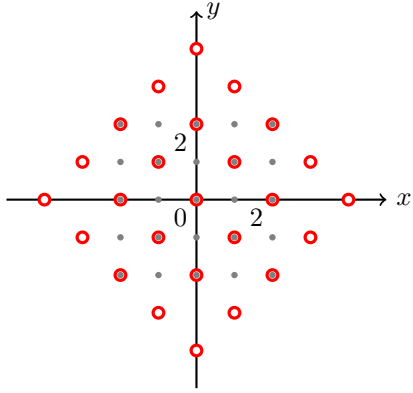

\begin{theorem}[Birkhoff's pointwise ergodic theorem] \label{Birkhoff}
$\;$
\begin{enumerate}
\item For all $(k,\ell) \in \mathbb{Z}^2_e$ such that $k\ne 0$ and $\ell \ne 0$, if $f \in L^1(\G)$, then
$$\displaystyle \lim\limits_{n \to +\infty} \frac{1}{n}\sum\limits_{i=0}^{n-1} f \circ \tau_{(k,l)}^{{i}} = \mathbb{E}(f) \mbox{ a.s.}$$
\item If $f \in L^1(\G)$, then
$\displaystyle \lim\limits_{N \to +\infty} \frac{1}{\#C_N}\sum\limits_{x \in C_N} f \circ \tau_{x} = \mathbb{E}(f)$ a.s.
\end{enumerate}
\end{theorem}

\begin{proof}
Since $\P$ is invariant under the even translation $\tau_{(k,l)}$, the first point is exactly Birkhoff's theorem (see for instance Theorem 2.1.5 of \cite{Keller}). By Proposition \ref{ergo}, the dynamical system $(\G,\P,\mathfrak{F},\tau_{(k,\ell)})$ is ergodic, thus the limit is constant.

In the second point, we apply a multi-dimensional ergodic theorem. Consider the two commuting even translations $\tau_1=\tau_{(1,1)}$ and $\tau_2=\tau_{(-1,1)}$. We can apply Theorem 2.1.5 of \cite{Keller} to the set $\mathcal{T}=\{\tau_1^{x_1}\circ \tau_2^{x_2}=\tau_{\phi(x_1,x_2)},~ (x_1,x_2)\in\Z^2\}$. The ergodicity of $(\G,\mathfrak{F},\mathbb{P},\tau_{1})$ ensures that the $\sigma$-algebra of the events invariant by $\mathcal{T}$ is trivial, so the limit is equal to~$\E(f)$.
\end{proof}

\section{Height function and infinite paths}

Now, we define the \emph{height} function $H: \Z^2+(\frac{1}{2},\frac{1}{2}) \to \mathbb Z$, which was the key ingredient introduced by G\'abor Pete in \cite{Corner} to study the case $(\frac{1}{2},\frac{1}{2})$. We color the faces of $\Z^2$ in a chessboard manner: a face $(n+\frac{1}{2},m+\frac{1}{2})$ is black if $n+m\in 2\Z$, otherwise it is white. Note that a path has only black faces along one side and only white faces along the other. We set $H\left(\frac{1}{2},\frac{1}{2}\right)=0$. For a face $(n+\frac{1}{2},m+\frac{1}{2})$, consider a path in $\Z^2+(\frac{1}{2},\frac{1}{2})$, from $(\frac{1}{2},\frac{1}{2})$ to $(n+\frac{1}{2},m+\frac{1}{2})$. We follow this path starting from $(\frac{1}{2},\frac{1}{2})$, and each time it crosses an open edge from a black face to a white face (respectively from a white face from a black face), we add (respectively substract) $1$ to the height. Note that it does not depend on the choice of the path, and all black faces (respectively all white faces) have the same height along a path. The \emph{level of a path} is the height of a black face along this path. 

Despite the strong dependencies of the model, the height function can be expressed as a function of two independent simple random walks on $\Z$. Remember that we denote by $(\xi(n))_{n\in \Z}$ (resp. $(\eta(m))_{m\in \Z}$) the values of the vertical (resp. horizontal) lines, and that its law is $(q\delta_1+(1-q)\delta_{-1})^{\otimes \mathbb{Z}}$ (resp. $(p\delta_1+(1-p)\delta_{-1})^{\otimes \mathbb{Z}}$). We define two independent random walks $(X_n)_{n \in \mathbb Z}$ and $(Y_m)_{m \in \mathbb Z}$ by setting:
\begin{align*}
    X_0=0, \quad \text{ for } n>0,\, X_n=\sum\limits_{i=1}^n \xi(i), & \quad \text{ and for }n<0,\, X_n=-\sum\limits_{i=n+1}^0 \xi(i);\\
    Y_0=0, \quad \text{ for } m>0,\, Y_m=\sum\limits_{i=1}^m \eta(i), & \quad \text{ and for }m<0,\, Y_m=-\sum\limits_{i=m+1}^0 \eta(i).
\end{align*}
It is not difficult to check that:
$$H\left(n+\frac{1}{2},m+\frac{1}{2}\right)=\left\lceil\frac{X_n+Y_m}{2}\right\rceil.$$
We can now prove our main result:

\begin{theorem} \label{existinftypath}
There almost surely exists an infinite number of infinite paths in corner percolation with parameters $(p,q)$, as soon as $(p,q)\ne(\frac{1}{2},\frac{1}{2})$. 
\end{theorem}

\begin{proof} 
\noindent\underline{Step 1.} We want to construct an event having positive probability, ensuring that the component of the origin is infinite. First of all, note that we have, for all $n\in \Z$,
$$\P(X_{n+1}=X_n+1)=q \text{\quad and\quad }\P(Y_{n+1}=Y_n+1)=p.$$
We start by assuming that $q>\frac{1}{2}$: by the symmetry properties $2$ and $3$ of Remark \ref{symmetry}, we will have the result for all parameters $(p,q)\ne (\frac{1}{2},\frac{1}{2})$. By the law of large numbers, we have $\lim\limits_{n\to +\infty}X_n=+\infty$ and $\lim\limits_{n\to +\infty} X_{-n}=-\infty$. Then let $K>\frac{16}{2q-1}$ be such that
$$\forall n\geq K,\quad \frac{X_n}{n}-\frac{2q-1}{2}>\frac{2q-1}{4}\quad \text{and}\quad \frac{X_{-n}}{n}-\frac{1-2q}{2}<\frac{1-2q}{4},$$
with positive probability. We have, for all $n\geq K$ and $m\in \N$ such that $m\leq \frac{2q-1}{2}n$, on one hand
\begin{align*}
    \left\lceil \frac{X_n+Y_m}{2} \right\rceil&\geq \frac{X_n-\frac{2q-1}{2}n}{2}
    \geq \frac{n}{2}\left(\frac{X_n}{n}-\frac{2q-1}{2}\right)
    >1,
    \end{align*}
and on the other hand
\begin{align*}
    \left\lceil \frac{X_{-n}+Y_{-m}}{2} \right\rceil&\leq \frac{X_{-n}+\frac{2q-1}{2}n}{2}+1
    \leq \frac{n}{2}\left(\frac{X_{-n}}{n}-\frac{1-2q}{2}\right)+1
    <-1.
    \end{align*}
Therefore, the event
$$A=\left\{\forall n\geq K,~\forall 0\leq m \leq \frac{2q-1}{2}n,\quad \left\lceil\frac{X_{n}+Y_{m}}{2}\right\rceil> 1 \text{ and }\left\lceil\frac{X_{-n}+Y_{-m}}{2}\right\rceil< -1\right\}$$
has positive probability. For $(\xi,\eta)\in \left(\{-1,1\}^{\Z}\right)^2$, we denote by $(\xi_K,\eta_K)\in \left(\{-1,1\}^{\llbracket -K,K \rrbracket}\right)^2$ the restrictions $\left((\xi_n)_{|n|\leq K},(\eta_n)_{|n|\leq K}\right)$.
Let us define the map 
$$\begin{array}{rcccl}
    \Pi & : & \left(\{-1,1\}^{\Z}\right)^2 & \to & \left(\{-1,1\}^{\Z\setminus \llbracket -K,K \rrbracket}\right)^2 \\
    & & (\xi,\eta) & \mapsto & \left((\xi_n)_{|n|>K},(\eta_n)_{|n|>K}\right)
\end{array}.$$
Now we set 
$$A_{St}=\left\{(\xi,\eta)\in (\{-1,1\}^{\Z})^2~:~(\xi_{K},\eta_{K})=(\mathbf{1}_K,\mathbf{1}_K),~\Pi(\xi,\eta)\in \Pi(A)\right\}.$$
As all vertical and horizontal lines are independent, and $A\subset \Pi^{-1}\circ\Pi(A)$, we have
\begin{align*}
\P\left(A_{St}\right)&\geq \P\left((\xi_{K},\eta_{K})=(\mathbf{1}_K,\mathbf{1}_K)\right)\P\left(\Pi^{-1}\circ\Pi(A)\right)\geq p^{(2K+1)^2}\P(A)>0.
\end{align*}
Now let us prove that $A_{St}\subset A$. Let $(\xi,\eta)\in A_{St}$: there exists $(\Tilde{\xi},\Tilde{\eta})\in A$ such that $\Pi(\Tilde{\xi},\Tilde{\eta})=\Pi(\xi,\eta)$. We denote by $(\Tilde{X}_n)_{n\in \Z}$ and $(\Tilde{Y}_m)_{m\in \Z}$ the same walks as before for $(\Tilde{\xi},\Tilde{\eta})$. Then for all $n\geq K$ and $0\leq m \leq \frac{2q-1}{2}n$, we have 
\begin{align*}
H_{(\xi,\eta)}\left(n+\frac{1}{2},m+\frac{1}{2}\right)&=\left\lceil \frac{X_n+Y_m}{2} \right\rceil\\
&=\left\lceil \frac{K+\sum\limits_{k=K+1}^{n}\xi(k)+\min(m,K)+\sum\limits_{k=\min(m,K)+1}^{m}\eta(k)}{2} \right\rceil\\
&=\left\lceil \frac{K+\sum\limits_{k=K+1}^{n}\Tilde{\xi}(k)+\min(m,K)+\sum\limits_{k=\min(m,K)+1}^{m}\Tilde{\eta}(k)}{2} \right\rceil\\
&\geq\left\lceil \frac{\sum\limits_{k=1}^{n}\Tilde{\xi}(k)+\sum\limits_{k=1}^{m}\Tilde{\eta}(k)}{2} \right\rceil\\
&\geq \left\lceil \frac{\Tilde{X}_n+\Tilde{Y}_m}{2} \right\rceil= H_{(\Tilde{\xi},\Tilde{\eta})}\left(n+\frac{1}{2},m+\frac{1}{2}\right).
\end{align*}
In the same manner we have, for all $n,m\leq -K$,
$$H_{(\xi,\eta)}\left(n+\frac{1}{2},m+\frac{1}{2}\right)\leq H_{(\Tilde{\xi},\Tilde{\eta})}\left(n+\frac{1}{2},m+\frac{1}{2}\right),$$
therefore $A_{St}\subset A$. We recall that $B_K=\llbracket -K,K \rrbracket^2$. On $A_{St}$, the components crossing $B_K$ are stairs in $B_K$, and all the heights of the faces in $B_K$ are known. The path containing the origin has level $-1$, so if it is finite, then the points $(-K,K)$ and $(K,-K)$ must be connected in $\Z^2\setminus B_K$. But it is not possible, because such a path would go through a zone where all the heights are positive or all the heights are negative, see Figure \ref{Prooftheorem8}.

\begin{figure}
    \centering
    \begin{tikzpicture}[scale=0.5]
    
    \fill[gray!20] (-4,2) -- (-3,2) -- (-3,3) -- (-4,3) -- cycle;
    \fill[gray!20] (-2,2) -- (-1,2) -- (-1,3) -- (-2,3) -- cycle;
    \fill[gray!20] (0,2) -- (1,2) -- (1,3) -- (0,3) --cycle;
    \fill[gray!20] (2,2) -- (3,2) -- (3,3) -- (2,3) --cycle;
    
    \fill[gray!20] (-3,2) -- (-2,2) -- (-2,1) -- (-3,1) -- cycle; 
    \fill[gray!20] (-1,2) -- (0,2) -- (0,1) -- (-1,1) -- cycle; 
    \fill[gray!20] (1,2) -- (2,2) -- (2,1) -- (1,1) -- cycle;
    \fill[gray!20] (3,2) -- (4,2) -- (4,1) -- (3,1) -- cycle;
    
    \fill[gray!20] (-3,4) -- (-2,4) -- (-2,3) -- (-3,3) -- cycle; 
    \fill[gray!20] (-1,4) -- (0,4) -- (0,3) -- (-1,3) -- cycle; 
    \fill[gray!20] (1,4) -- (2,4) -- (2,3) -- (1,3) -- cycle;
    \fill[gray!20] (3,4) -- (4,4) -- (4,3) -- (3,3) -- cycle;
    
    \fill[gray!20] (-4,0) -- (-3,0) -- (-3,1) -- (-4,1) -- cycle;
    \fill[gray!20] (-2,0) -- (-1,0) -- (-1,1) -- (-2,1) -- cycle;
    \fill[gray!20] (0,0) -- (1,0) -- (1,1) -- (0,1) --cycle;
    \fill[gray!20] (2,0) -- (3,0) -- (3,1) -- (2,1) --cycle;
    
    \fill[gray!20] (-3,0) -- (-2,0) -- (-2,-1) -- (-3,-1) -- cycle; 
    \fill[gray!20] (-1,0) -- (0,0) -- (0,-1) -- (-1,-1) -- cycle; 
    \fill[gray!20] (1,0) -- (2,0) -- (2,-1) -- (1,-1) -- cycle;
    \fill[gray!20] (3,0) -- (4,0) -- (4,-1) -- (3,-1) -- cycle;
    
    \fill[gray!20] (-4,-2) -- (-3,-2) -- (-3,-1) -- (-4,-1) -- cycle;
    \fill[gray!20] (-2,-2) -- (-1,-2) -- (-1,-1) -- (-2,-1) -- cycle;
    \fill[gray!20] (0,-2) -- (1,-2) -- (1,-1) -- (0,-1) --cycle;
    \fill[gray!20] (2,-2) -- (3,-2) -- (3,-1) -- (2,-1) --cycle;
    
    \fill[gray!20] (-3,-2) -- (-2,-2) -- (-2,-3) -- (-3,-3) -- cycle; 
    \fill[gray!20] (-1,-2) -- (0,-2) -- (0,-3) -- (-1,-3) -- cycle; 
    \fill[gray!20] (1,-2) -- (2,-2) -- (2,-3) -- (1,-3) -- cycle;
    \fill[gray!20] (3,-2) -- (4,-2) -- (4,-3) -- (3,-3) -- cycle;
    
    \fill[gray!20] (-4,-4) -- (-3,-4) -- (-3,-3) -- (-4,-3) -- cycle;
    \fill[gray!20] (-2,-4) -- (-1,-4) -- (-1,-3) -- (-2,-3) -- cycle;
    \fill[gray!20] (0,-4) -- (1,-4) -- (1,-3) -- (0,-3) --cycle;
    \fill[gray!20] (2,-4) -- (3,-4) -- (3,-3) -- (2,-3) --cycle;
    
    \draw[black, thin] (-4,1) -- (-3,1);
    \draw[black, thin] (-2,1) -- (-1,1);
    \draw[black, thin] (0,1) -- (1,1);
    \draw[black, thin] (2,1) -- (3,1);
    
    \draw[black, thin] (-3,2) -- (-2,2);
    \draw[black, thin] (-1,2) -- (0,2);
    \draw[black, thin] (1,2) -- (2,2);
    \draw[black, thin] (3,2) -- (4,2);
    
    \draw[black, thin] (-4,3) -- (-3,3);
    \draw[black, thin] (-2,3) -- (-1,3);
    \draw[black, thin] (0,3) -- (1,3);
    \draw[black, thin] (2,3) -- (3,3);
    
    \draw[black, thin] (-3,-2) -- (-2,-2);
    \draw[black, thin] (-1,-2) -- (0,-2);
    \draw[black, thin] (1,-2) -- (2,-2);
    \draw[black, thin] (3,-2) -- (4,-2);
    
    \draw[black, thin] (-4,-1) -- (-3,-1);
    \draw[black, thin] (-2,-1) -- (-1,-1);
    \draw[black, thin] (0,-1) -- (1,-1);
    \draw[black, thin] (2,-1) -- (3,-1);
    
    \draw[black, thin] (-3,0) -- (-2,0);
    \draw[black, thin] (-1,0) -- (0,0);
    \draw[black, thin] (1,0) -- (2,0);
    \draw[black, thin] (3,0) -- (4,0);
    
    \draw[black, thin] (-4,-3) -- (-3,-3);
    \draw[black, thin] (-2,-3) -- (-1,-3);
    \draw[black, thin] (0,-3) -- (1,-3);
    \draw[black, thin] (2,-3) -- (3,-3);
    
    \draw[black, thin] (-3,-4) -- (-2,-4);
    \draw[black, thin] (-1,-4) -- (0,-4);
    \draw[black, thin] (1,-4) -- (2,-4);
    \draw[black, thin] (3,-4) -- (4,-4);
    
    \draw[black, thin] (-3,4) -- (-2,4);
    \draw[black, thin] (-1,4) -- (0,4);
    \draw[black, thin] (1,4) -- (2,4);
    \draw[black, thin] (3,4) -- (4,4);

    \draw[black, thin] (-3,3) -- (-3,2);
    \draw[black, thin] (-3,1) -- (-3,0);
    \draw[black, thin] (-3,-1) -- (-3,-2);
    \draw[black, thin] (-3,-3) -- (-3,-4);
    
    \draw[black, thin] (-1,3) -- (-1,2);
    \draw[black, thin] (-1,1) -- (-1,0);
    \draw[black, thin] (-1,-1) -- (-1,-2);
    \draw[black, thin] (-1,-3) -- (-1,-4);
    
    \draw[black, thin] (0,4) -- (0,3);
    \draw[black, thin] (0,2) -- (0,1);
    \draw[black, thin] (0,0) -- (0,-1);
    \draw[black, thin] (0,-2) -- (0,-3);
    
    \draw[black, thin] (1,3) -- (1,2);
    \draw[black, thin] (1,1) -- (1,0);
    \draw[black, thin] (1,-1) -- (1,-2);
    \draw[black, thin] (1,-3) -- (1,-4);
    
    \draw[black, thin] (3,3) -- (3,2);
    \draw[black, thin] (3,1) -- (3,0);
    \draw[black, thin] (3,-1) -- (3,-2);
    \draw[black, thin] (3,-3) -- (3,-4);
    
    \draw[black, thin] (2,4) -- (2,3);
    \draw[black, thin] (2,2) -- (2,1);
    \draw[black, thin] (2,0) -- (2,-1);
    \draw[black, thin] (2,-2) -- (2,-3);
    
    \draw[black, thin] (4,4) -- (4,3);
    \draw[black, thin] (4,2) -- (4,1);
    \draw[black, thin] (4,0) -- (4,-1);
    \draw[black, thin] (4,-2) -- (4,-3);
    
    \draw[black, thin] (-2,4) -- (-2,3);
    \draw[black, thin] (-2,2) -- (-2,1);
    \draw[black, thin] (-2,0) -- (-2,-1);
    \draw[black, thin] (-2,-2) -- (-2,-3);
    
    \draw[black, thin] (-4,4) -- (-4,3);
    \draw[black, thin] (-4,2) -- (-4,1);
    \draw[black, thin] (-4,0) -- (-4,-1);
    \draw[black, thin] (-4,-2) -- (-4,-3);
    
    \filldraw[black] (1/2,1/2) circle (0pt) node[]{$0$}; 
    \filldraw[black] (-1/2,3/2) circle (0pt) node[]{$0$}; 
    \filldraw[black] (3/2,-1/2) circle (0pt) node[]{$0$}; 
    \filldraw[black] (5/2,-3/2) circle (0pt) node[]{$0$}; 
    \filldraw[black] (7/2,-5/2) circle (0pt) node[]{$0$}; 
    \filldraw[black] (-3/2,5/2) circle (0pt) node[]{$0$}; 
    \filldraw[black] (-5/2,7/2) circle (0pt) node[]{$0$}; 
    
    \filldraw[black] (3/2,3/2) circle (0pt) node[]{$1$}; 
    \filldraw[black] (5/2,1/2) circle (0pt) node[]{$1$}; 
    \filldraw[black] (7/2,-1/2) circle (0pt) node[]{$1$}; 
    \filldraw[black] (1/2,5/2) circle (0pt) node[]{$1$}; 
    \filldraw[black] (-1/2,7/2) circle (0pt) node[]{$1$}; 
    
    \filldraw[black] (-7/2,5/2) circle (0pt) node[]{$-1$};
    \filldraw[black] (5/2,-7/2) circle (0pt) node[]{$-1$};
    \filldraw[black] (3/2,-5/2) circle (0pt) node[]{$-1$};
    \filldraw[black] (1/2,-3/2) circle (0pt) node[]{$-1$};
    \filldraw[black] (-5/2,3/2) circle (0pt) node[]{$-1$};
    \filldraw[black] (-3/2,1/2) circle (0pt) node[]{$-1$};
    \filldraw[black] (-1/2,-1/2) circle (0pt) node[]{$-1$};
    \filldraw[black] (-3/2,1/2) circle (0pt) node[]{$-1$};
    \filldraw[black] (1/2,-3/2) circle (0pt) node[]{$-1$};
     
    \filldraw[black] (-3/2,-3/2) circle (0pt) node[]{$-2$};
    \filldraw[black] (-5/2,-1/2) circle (0pt) node[]{$-2$};
    \filldraw[black] (-7/2,1/2) circle (0pt) node[]{$-2$};
    \filldraw[black] (-1/2,-5/2) circle (0pt) node[]{$-2$};
    \filldraw[black] (1/2,-7/2) circle (0pt) node[]{$-2$};
    
    \filldraw[black] (-5/2,-5/2) circle (0pt) node[]{$-3$};
    \filldraw[black] (-7/2,-3/2) circle (0pt) node[]{$-3$};
    \filldraw[black] (-3/2,-7/2) circle (0pt) node[]{$-3$};
    
    \filldraw[black] (5/2,5/2) circle (0pt) node[]{$2$};
    \filldraw[black] (7/2,3/2) circle (0pt) node[]{$2$};
    \filldraw[black] (3/2,7/2) circle (0pt) node[]{$2$};
    
    \filldraw[black] (-7/2,-7/2) circle (0pt) node[]{$-4$};
    
    \filldraw[black] (7/2,7/2) circle (0pt) node[]{$3$};
    
    \filldraw[red] (0,0) circle (2pt); 
    \filldraw[blue] (-4,4) circle (2pt); 
    \filldraw[blue] (4,-4) circle (2pt); 
    
    \draw[black, thick] (-4,4/5) -- (-10,2) node[anchor=south west]{$y=\frac{-(2q-1)}{2}x$};
    \draw[black, thick] (-4,-4/5) -- (-10,-2) node[anchor=north west]{$y=\frac{2q-1}{2}x$};
    
    \draw[black, thick] (4,4/5) -- (10,2) node[anchor=south east]{$y=\frac{2q-1}{2}x$};
    \draw[black, thick] (4,-4/5) -- (10,-2) node[anchor=north east]{$y=\frac{-(2q-1)}{2}x$};

    \filldraw[black] (6,0) circle (0pt) node[]{$C_1$};
    \filldraw[black] (-6,0) circle (0pt) node[]{$C_{-1}$};

    \end{tikzpicture}
    \caption{Example of the construction of the event $A_{St}$ for $q=0.7$ and $K=4$. The path containing the origin has level $-1$. If a point $(n,m)$ is in $C_{-1}$ or $C_1$, then we have $H(n+\frac{1}{2},m+\frac{1}{2})>1$ or $H(n+\frac{1}{2},m+\frac{1}{2})<-1$. The only face with height $-1$ on the $x$-axis of $\Z^2+(\frac{1}{2},\frac{1}{2})$ is $(-\frac{3}{2},\frac{1}{2})$. Therefore, the two blue points cannot be in the same finite cycle.}
    \label{Prooftheorem8}
\end{figure}
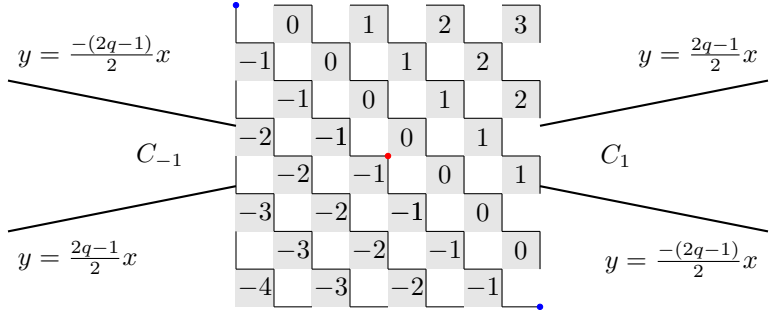

Therefore, on the event $A_{St}$, which has positive probability, the connected component of the origin is infinite.

\medskip \noindent\underline{Step 2.} Now we use Theorem \ref{Birkhoff} to construct an infinity of infinite paths. Let $(\alpha_1,\alpha_2)\in \N\times \N^*$ be such that $0<\frac{\alpha_1}{\alpha_2}<\frac{2q-1}{2}$. By Theorem \ref{Birkhoff}, we have
$$\lim\limits_{n \to +\infty} \frac{1}{n}\sum\limits_{i=0}^{n-1} \mathds{1}_{A_{St}} \circ \tau_{(\alpha_2,\alpha_1)}^{{i}} = \P(A_{St})>0 ~a.s.,$$
which means that almost surely, there exists a sequence of integers $(N_k)_{k\in \N}$ such that for all $k\in \mathbb{N}$, $\tau^{-{N_k}}(A_{St})$ is realised and $N_{k+1}-N_k>K$. The components of the points $((N_k,N_k))_{k\in \N}$ are infinite, and the last condition ensures that they are disjoint two by two, since they all have different levels. Therefore, we just proved that there almost surely exists an infinite number of infinite connected components.
\end{proof}

In the next section, we prove the last part of Theorem \ref{PremierResultat}, stating that all infinite paths have a same asymptotic slope. The proof will require one more time the height function.

\section{Direction of the infinite paths}


\begin{figure}[h]
    \centering
    \begin{subfigure}{.45\textwidth}
    \includegraphics[width=5.2cm,height=5.2cm]{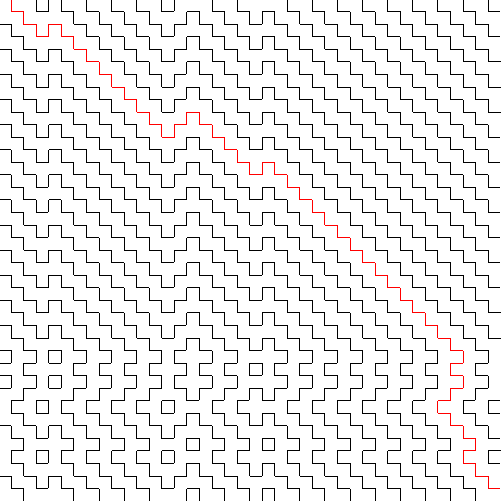}
    \end{subfigure}\hspace{0.5cm}
    \begin{subfigure}{.45\textwidth}
    \includegraphics[width=5.2cm,height=5.2cm]{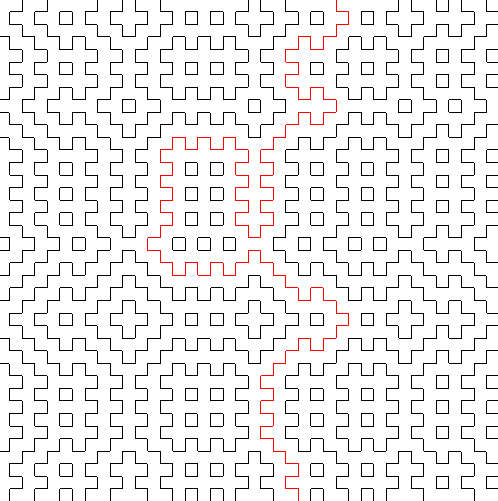}
    \end{subfigure}
    \newline

   \vspace{0.2cm}
    \begin{subfigure}{.45\textwidth}
    \includegraphics[width=5.2cm,height=5.2cm]{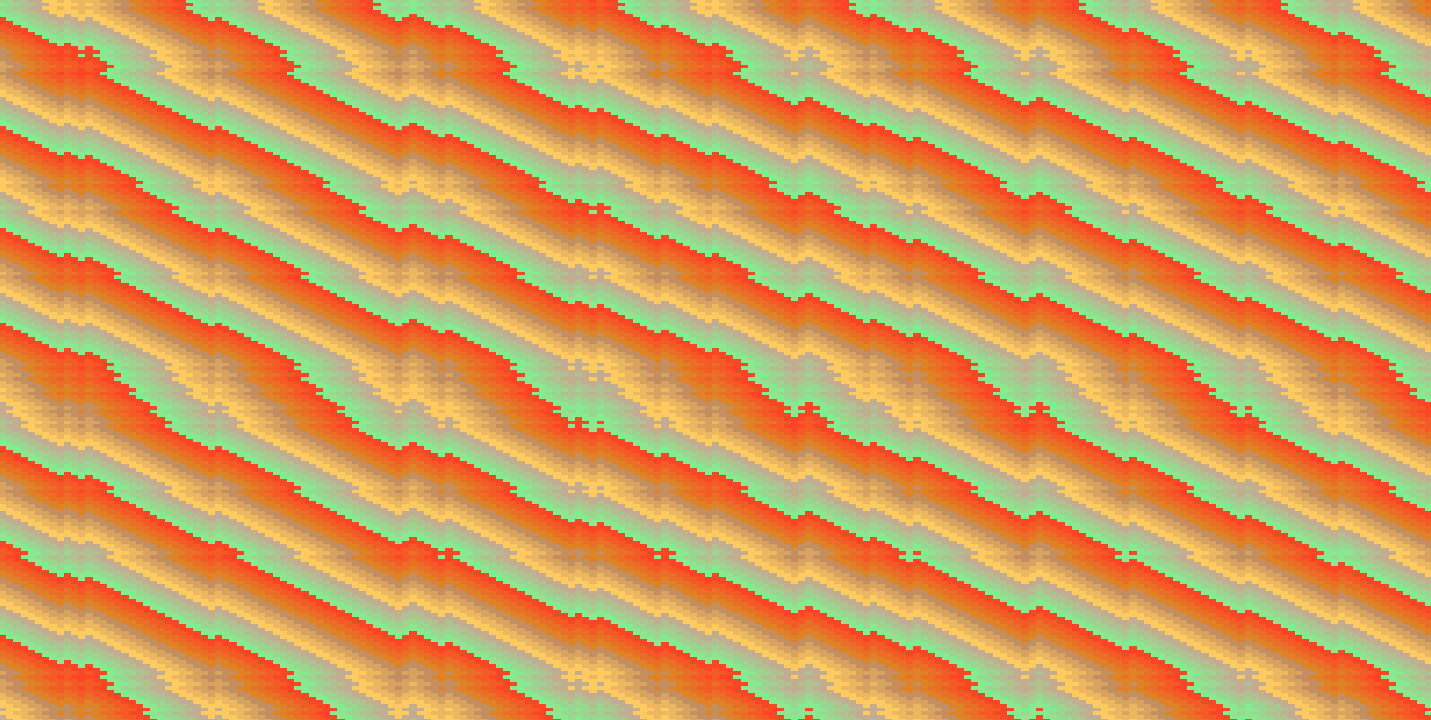}
    \end{subfigure}\hspace{0.5cm}
    \begin{subfigure}{.45\textwidth}
    \includegraphics[width=5.2cm,height=5.2cm]{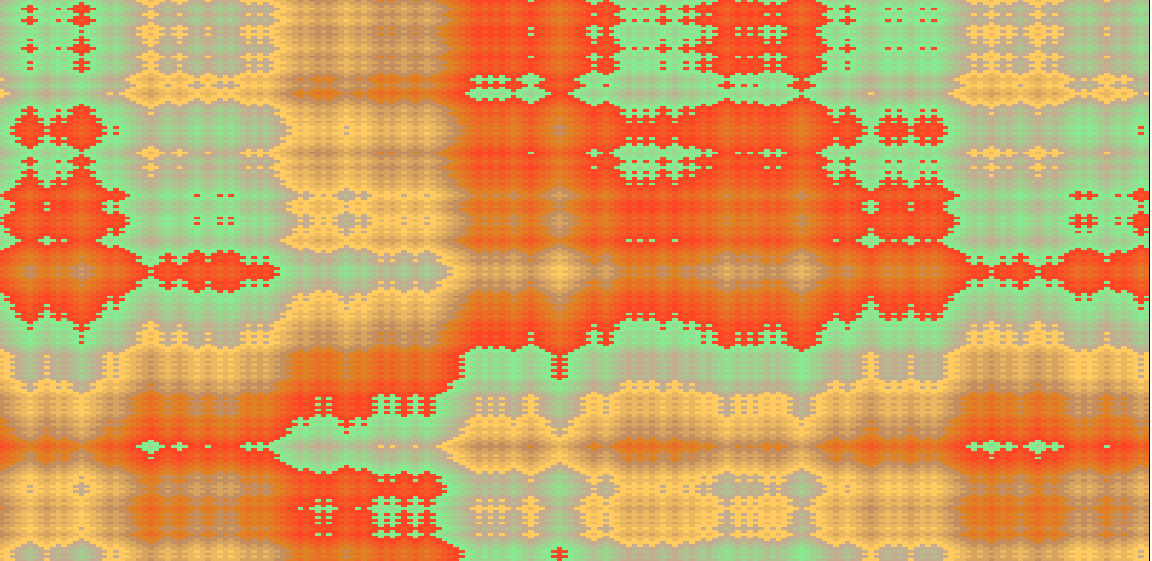}
    \end{subfigure}
    \caption{Simulations on the left (resp. right) have parameters $p=q=0.9$ (resp. $q=0.6$ and $p=0.5$). Faces are colored acccording to their heights, with colors that repeat periodically.
    }
    \label{SimuEscaliers}
\end{figure}

Remember that the connected components of the graph are either finite circuits or bi-infinite simple paths. We encode the connected component of the origin by a bi-infinite, possibly periodic, path $\Gamma=(\Gamma_{n})_{n \in \mathbb Z}$ of vertices. We choose its orientation by taking~$\Gamma_1$ such that when we walk along the path in the positive direction, the face with the highest height is on the right side of the path. Note that it amounts to choosing $\Gamma_1$ such that $\{\Gamma_0,\Gamma_1\}$ is a horizontal edge. The origin of the path is $\Gamma_0=(0,0)$. If the connected component of the origin is infinite, then $\Gamma$ is simple, while if the connected component of the origin is a circuit, then $\Gamma$ is periodic.

Let us denote by $D_R$ the disk of radius $R\in \mathbb{R}^*_+$, and by $\Z[i]$ the set $\{a+ib~:~ a,b\in \Z \}$. We set $z_{p,q}=2p-1+i(2q-1)$. We denote by $\theta_{p,q}$ the argument of $z_{p,q}$ in $\left[0,2\pi\right[$. Then we define, for $\varepsilon>0$, the following discrete cone
$$C_{p,q}^\varepsilon=\{z\in \Z[i]~ :~ \langle \frac{z}{|z|},e^{i\theta_{p,q}} \rangle  \leq \varepsilon\}.$$
Note that $\tan\left(\theta_{p,q}+\frac{\pi}{2}\right)=\frac{2q-1}{1-2p}$, if $p\ne \frac{1}{2}$. Therefore, $C_{p,q}^\varepsilon$ is centered around the line of equation $y=\frac{2q-1}{1-2p}x$ if $p\ne \frac{1}{2}$, and around the $y$-axis if $p=\frac{1}{2}$.

\begin{theorem} \label{directionofpath}
Let $(p,q)\ne (\frac{1}{2},\frac{1}{2})$. Then we have, for all $\varepsilon>0$,
$$\P(\exists R\in \mathbb{R}^*_+~:~\Gamma \subset D_R \cup C_{p,q}^\varepsilon)=1.$$
\end{theorem}

This implies that if $\Gamma$ is simple, then $\Gamma$ has an asymptotic slope, which is $\frac{2q-1}{1-2p}$. By invariance by even translation, the same result holds for all bi-infinite paths.

\begin{figure}[h]
    \centering
    \begin{tikzpicture}[scale=0.3]
    
    \draw[black, thick, ->] (0,0) -- (0,5) node[anchor=west]{$y$};
    \draw[black, thick] (0,-5) -- (0,0);
    \draw[black, thick, ->] (-15,0) -- (15,0) node[anchor=west]{$x$};
    
    \fill[blue, thin] (0,0) -- (-14,5) -- (-14,2.5) -- (14,-2.5) -- (14,-5) -- cycle;
    
    \fill[red] circle(2);
    
    \filldraw[red] (0,0) circle (2pt) node[anchor=south east]{};
    
    \filldraw[blue] (-14,2.5) node[anchor=north]{$C_{p,q}^\varepsilon$};
    
    \filldraw[blue] (14,-2.2) node[anchor=south east]{$C_{p,q}^\varepsilon$};
    
    \filldraw[red] (-2,0) node[anchor=north east]{$D_R$};
        
    \end{tikzpicture}
    \caption{Representation of the set described in Theorem \ref{directionofpath}. The path $\Gamma$ is included in the colored area for some value of $R\in \mathbb{R}_+^*$.}
    \label{BoiteEtCone}
\end{figure}
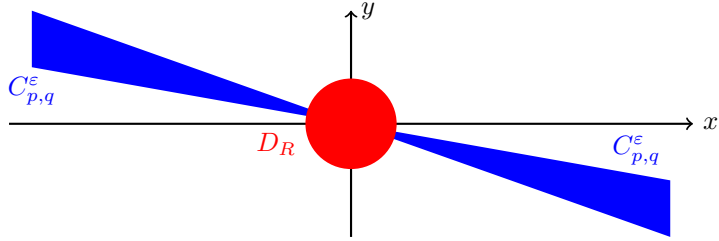

\begin{proof}
Note that $\Gamma$ has level $0$ or $-1$. This implies that the faces adjacent to $\Gamma$ have height $0,1$ or $-1$.  Using the asymptotic behaviour of the random walks, we will see that almost surely, these faces are all in a restricted part of the grid.

The key step to prove the theorem is to show that, for all $\varepsilon>0$ and $0<\varepsilon_0<\frac{|z_{p,q}|\varepsilon}{2}$, almost surely there exists $R_0\in \mathbb{R}^*_+$ such that, for all $\theta\in [0,2\pi[$ and $r\geq R_0$ satisfying $(r\cos(\theta),r\sin(\theta))\in \Z^2$,
\begin{align}
    \left | \frac{X_{r\cos(\theta)}+Y_{r\sin(\theta)}}{r}- \langle z_{p,q},e^{i\theta} \rangle_{\mathbb{R}^2}\right|\leq \varepsilon_0.\label{eq1}
\end{align}
Let us start by showing that $(\ref{eq1})$ implies the result. We set $\varepsilon>0$ and $R=\max\left(R_0,\frac{4}{|z_{p,q}|\varepsilon}\right)$. Let $(n,m)\in \Z^2\setminus D_R$ be such that $(n,m)\notin C_{p,q}^\varepsilon$. There exist $r\in \mathbb{R}^*_+$ and $\theta\in [0,2\pi[$ such that $(n,m)=(r\cos(\theta),r\sin(\theta))$, with $r>R$ and $\left |\langle e^{i\theta_{p,q}},e^{i\theta} \rangle_{\mathbb{R}^2}\right| > \varepsilon$. Then we have $\left|\langle z_{p,q},e^{i\theta} \rangle_{\mathbb{R}^2}\right| >|z_{p,q}|\varepsilon$, and
\begin{align*}
    \left |H\left(n+\frac{1}{2},m+\frac{1}{2}\right)\right|&\geq\frac{\left | X_{r\cos(\theta)}+Y_{r\sin(\theta)}\right|}{2}\\
    &\geq\frac{r}{2}\left|\; \left| \frac{\left| X_{r\cos(\theta)}+Y_{r\sin(\theta)}\right|}{r}-\langle z_{p,q},e^{i\theta} \rangle_{\mathbb{R}^2} \right|-\left|\langle z_{p,q},e^{i\theta} \rangle_{\mathbb{R}^2} \right| \; \right|\\
    &\geq \frac{r}{2}\left| \; \left|\langle z_{p,q},e^{i\theta} \rangle_{\mathbb{R}^2}\right| -\varepsilon_0\right|\\
    &\geq \frac{r}{2}\left(|z_{p,q}|\varepsilon-\varepsilon_0\right)\\
    &\geq \frac{r|z_{p,q}|\varepsilon}{4}>1,
\end{align*}
the fourth and fifth inequalities coming from the fact that $\frac{|z_{p,q}|\varepsilon}{2}>\varepsilon_0$. Since the level of~$\Gamma$ is $0$ or $-1$, then we deduce that $(n,m)\notin \Gamma$. Therefore, $\Gamma\subset D_R\cup C_{p,q}^\varepsilon$.

Now, let us prove (\ref{eq1}). By the law of large numbers, there exists $N\in \N^*$ such that for all $n\geq N$, 
\begin{align}
\max\left(\left|\frac{X_{n}}{n}-(2p-1)\right|, \left|\frac{X_{-n}}{n}-(1-2p)\right|, \left|\frac{Y_{ n}}{n}-(2q-1)\right|,  \left|\frac{Y_{ -n}}{n}-(1-2q)\right|\right)\leq \frac{\varepsilon_0}{2}. \label{eq2}
\end{align}
We set $R_0=\frac{4(N+1)}{\varepsilon_0}$. For $\theta\in [0,2\pi[$ and $r\geq R_0$ such that $(r\cos(\theta),r\sin(\theta))\in \Z^2$, we have
\begin{align}
&\left | \frac{X_{r\cos(\theta)}+Y_{r\sin(\theta)}}{r}-\left [(2p-1)\cos(\theta)+(2q-1)\sin(\theta)\right]\right|\notag \\
&\leq \left | \frac{X_{r\cos(\theta)}}{r}-(2p-1)\cos(\theta)\right|+\left | \frac{Y_{r\sin(\theta)}}{r}-(2q-1)\sin(\theta)\right|. \label{eq3}
\end{align}
We now bound the first term from above. We have two cases:

\medskip \noindent\underline{Case 1:} $|\cos(\theta)|\leq \frac{\varepsilon_0}{4}$. \\
In this case, we bound from above $X_{r\cos(\theta)}$ with the triangular inequality:
\begin{align*}
    \left| \frac{X_{ r\cos(\theta)}}{r}\right|&\leq |\cos(\theta)|.
\end{align*}
Since $|(2p-1)\cos(\theta)|\leq |\cos(\theta)|$, we have
\begin{align*}
    \left| \frac{X_{ r\cos(\theta)}}{r}- (2p-1)\cos(\theta)\right|&\leq 2\cos(\theta)\leq \frac{\varepsilon_0}{2}.
\end{align*}

\medskip \noindent\underline{Case 2:} $|\cos(\theta)|>\frac{\varepsilon_0}{4}$. \\
This time we have $r|\cos(\theta)| \geq N$, so by \eqref{eq2}, we have
$$\left|\frac{X_{r\cos(\theta)}}{r\cos(\theta)}-(2p-1)\right|\leq \frac{\varepsilon_0}{2}.$$
Then we have
\begin{align*}
    \left| \frac{X_{ r\cos(\theta)}}{r}- (2p-1)\cos(\theta)\right|
    &\leq \left|\frac{X_{ r\cos(\theta)}}{r\cos(\theta) }\cos(\theta)- (2p-1)\cos(\theta)\right|\\
    &\leq |\cos(\theta)|\left| \frac{X_{ r\cos(\theta)}}{r\cos(\theta) }-(2p-1)\right| \leq \frac{\varepsilon_0}{2}.
\end{align*}
Working in the same manner for the second term in \eqref{eq3}, we obtain \eqref{eq1}.
\end{proof}

\begin{remark} \label{Rebroussement}
Remember that $\Gamma$ is a bi-infinite path. We define the \emph{forward path} $\Gamma^+$ (resp. the \emph{backward path} $\Gamma^-$) by
$$\Gamma^+=(\Gamma_n)_{n\in \N} \qquad \text{(resp. }\Gamma^-=(\Gamma_{-n})_{n\in \N}).$$
Note that Theorem \ref{directionofpath} implies that, conditionally on the event $\{\Gamma \text{ is simple}\}$, we have
$$\lim\limits_{n\to +\infty} \frac{\Gamma^+_n}{|\Gamma^+_n|}\in \{e^{i\Tilde{\theta}_{p,q}},e^{-i\Tilde{\theta}_{p,q}}\}\text{\quad and\quad} \lim\limits_{n\to +\infty} \frac{\Gamma^-_n}{|\Gamma^-_n|}\in \{e^{i\Tilde{\theta}_{p,q}},e^{-i\Tilde{\theta}_{p,q}}\},$$
where $\Tilde{\theta}_{p,q}=\theta_{p,q}+\frac{\pi}{2}$ (modulo $2\pi$).

In the proof of Theorem \ref{existinftypath}, we proved that with positive probability, we have 
$$\lim\limits_{n\to +\infty} \frac{\Gamma^+_n}{|\Gamma^+_n|}\ne\lim\limits_{n\to +\infty} \frac{\Gamma^-_n}{|\Gamma^-_n|}.$$
Indeed, on the event we constructed (see Figure \ref{Prooftheorem8}), vertices of $\Gamma^+$ (resp. $\Gamma^-$) are restricted to a certain area after a sufficient number of steps. 

Using Theorem $2$ of Burton and Keane \cite{Burton-Keane}, we can prove that in fact, forward and backward paths always have the same asymptotic direction. In other words, we have almost surely $\lim\limits_{n\to +\infty} \frac{\Gamma^+_n}{|\Gamma^+_n|}\ne\lim\limits_{n\to +\infty} \frac{\Gamma^-_n}{|\Gamma^-_n|}$. Indeed, Theorem $2$ of \cite{Burton-Keane} asserts that an infinite connected component $C$ has exactly one neighbor on each connected component of $\Z^2\setminus C$, ignoring finite cycles.
If we assume by contradiction that forward and backward paths may go to the same direction, then we deduce that an event of the type of the one of Figure \ref{Trifurcation} has positive probability, which leads to a contradiction.

\begin{figure}
    \centering
    \includegraphics[scale=0.5]{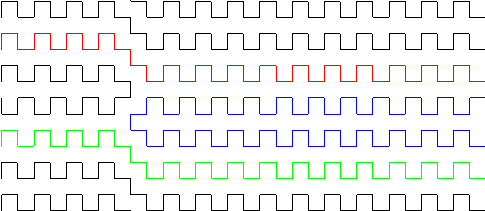}
    \caption{Example of a configuration containing a bi-infinite path whose forward and backward paths go to the same direction (the blue one). This path has two neighbors on the same side (the red one and the green one).}
    \label{Trifurcation}
\end{figure}


\end{remark}

\section{Questions and discussions}

\subsection{Level of infinite paths}


On Figure \ref{SimuEscaliers}, it seems that two distinct infinite connected components have different levels, and that each integer is the level of one infinite connected component. This leads to the following conjecture:

\begin{conjecture} The function which maps each infinite connected component to its level is bijective. 
Moreover, two infinite connected components are neighbors if and only if the difference of their level is $-1$ or $1$.
\end{conjecture}

\subsection{Ergodicity of the processes of the signs}

Remember that horizontal edges and vertical edges alternate along $\Gamma$, and the edges $E_{2n}=\{\Gamma_{2n},\Gamma_{2n+1}\}$ (resp. $E_{2n+1}=\{\Gamma_{2n+1},\Gamma_{2n+2}\}$), $n\in \mathbb Z$, are horizontal (resp. vertical). We define the sign of a horizontal (resp. vertical) edge by setting
$$\Gamma_{2n+1}-\Gamma_{2n}=-\sgn(E_{2n})e_1 \qquad \text{ (resp. } \Gamma_{2n+2}-\Gamma_{2n+1}=\sgn(E_{2n+1})e_2).$$
 With the help of Theorem \ref{Birkhoff}, we can prove the following:

\begin{proposition} \label{statio}
Let $p,q\in ]0,1[$. The process $(\sgn(E_{2n}))_{n\in \mathbb{Z}}$ (resp. $(\sgn(E_{2n+1}))_{n\in \mathbb{Z}}$) is \emph{stationary}, that is:
$$\mathbb{P}(\sgn(E_{t_1})=x_1,...,\sgn(E_{t_n})=x_n)=\mathbb{P}(\sgn(E_{t_1+k})=x_1,...,\sgn(E_{t_n+k})=x_n)$$
for all $n\in \mathbb{N}^*$, $x_1,x_2,...,x_n \in \{-1,1\}^n$,~ $k \in 2\mathbb{Z}$ ~and~ $t_1<t_2<...<t_n$ in  $2\mathbb{Z}$ (resp. in $2\mathbb{Z}+1$).
In particular, we have 
$$\forall n\in \Z,\quad \mathbb{E}(\sgn(E_{2n}))= 2p-1 \quad \text{and}\quad \mathbb{E}(\sgn(E_{2n+1}))= 2q-1.$$ 
\end{proposition}

If $p\ne \frac12$ and $q \ne \frac12$, bi-infinite simple paths have an asymptotic direction by Theorem~\ref{directionofpath}. This implies that there exists a random integer $N$ such that $(\Gamma_n)_{n\geq N}$ does not visit any line or column visited by $(\Gamma_n)_{0\leq n\leq k}$. Lines and columns being independent, the dynamical system seems to have the mixing property, which would imply its ergodicity, but we did not manage to prove it. 

\begin{conjecture} \label{ConjErgo}
Assume that $p\ne \frac12$ and $q \ne \frac12$. Let us denote by $X$ a process among $(\sgn(E_{2n}))_{n\in \mathbb{Z}}$ and $(\sgn(E_{2n+1}))_{n\in \mathbb{Z}}$. Conditionally on the event ``$\Gamma$ is simple'', the process~$X$ is \emph{ergodic}, that is:
$(\{-1,1\}^\N,\mathbb{Q}_X,\theta)$ is ergodic, where
$\theta$ is the shift map, $\mathbb{Q}=\P(.|\Gamma \text{ is simple)}$ and $\mathbb{Q}_X$ is the law of $X$ under $\mathbb{Q}$.
\end{conjecture}

We denote by $A$ the event ``$\Gamma \text{ is simple}$''. This conjecture would imply that
$$\lim\limits_{n\to +\infty}  \frac{1}{n}\sum\limits_{i=0}^{n-1}\sgn(E_{2i})=\frac{2p-1}{\P(A)}\mathds{1}_{A}\text{\quad and\quad} \lim\limits_{n\to +\infty}  \frac{1}{n}\sum\limits_{i=0}^{n-1}\sgn(E_{2i+1})=\frac{2q-1}{\P(A)}\mathds{1}_{A},$$
and the same result for the backward path. We would also obtain the following improvement of Remark \ref{Rebroussement}:
$$\P_{A}\left(\lim\limits_{n\to +\infty} \frac{\Gamma_n^+}{|\Gamma_n^+|}=e^{i\Tilde{\theta}_{p,q}} \text{\quad and\quad}\lim\limits_{n\to +\infty} \frac{\Gamma_n^-}{|\Gamma_n^-|}=e^{-i\Tilde{\theta}_{p,q}}\right)=1.$$

\subsection{Distribution of the sequences $\xi$ and $\eta$}

We assumed the sequences $\xi$ and $\eta$ to be i.i.d. But in fact, we only used the three following properties:
\begin{itemize}
    \item the corner percolation is ergodic for the translations of Proposition \ref{ergo};
    \item one of the limits of $(\frac{X_n}{n})_{n\in \N}$ and $(\frac{Y_n}{n})_{n\in \N}$ is non-zero;
    \item $\mu^{p,q}$ has a finite energy property.
\end{itemize}
So we could take more general distributions for $\xi$ and $\eta$, as long as these properties are still satisfied. In particular, all the results could be extended to two independent ergodic stationary Markov chains on $\{-1,1\}$. 





\end{document}